\documentclass[11pt]{article} 
\usepackage[utf8]{inputenc}
\usepackage[T1]{fontenc}
\usepackage[french,english]{babel}

\usepackage{amsmath,amsfonts,amssymb,amsthm,mathrsfs,amsrefs, bbm}

\renewcommand{\leq}{\leqslant}
\renewcommand{\geq}{\geqslant}
\usepackage[cmintegrals,libertine]{newtxmath}
\usepackage[cal=euler]{mathalfa}
\usepackage[mono=false]{libertine}
\useosf
\usepackage[a4paper,vmargin={3.5cm,3.5cm},hmargin={2.5cm,2.5cm}]{geometry}
\usepackage[font={small,sf}, labelfont={sf,bf}, margin=1cm]{caption}
\usepackage[pdftex,colorlinks=true]{hyperref}
\usepackage[expansion]{microtype}
\usepackage[pdftex]{color,graphicx}

\usepackage{stackrel}
\usepackage{soul} %\st{} pour barrer

\theoremstyle{plain}
\newtheorem{theorem}{Theorem}
\newtheorem{conjecture}{Conjecture}

\newtheorem{proposition}[theorem]{Proposition}

\theoremstyle{definition}
\newtheorem{definition}[theorem]{Definition}

\newcommand{\E}{\mathbb{E}}

\newcommand{\PP}{\mathbb{P}}

\newcommand{\kF}{\mathcal{F}}

\newcommand{\eps}{\varepsilon}

\newcommand{\hyp}{\mathrm{hyp}}

\hypersetup{
    pdftitle    = {}}

\usepackage{overpic}

\title{The diameter of random Bely\u{\i} surfaces}
\author{\textsc{Thomas Budzinski}\footnote{ENS Paris and Universit\'e Paris-Saclay. E-mail: \href{mailto:thomas.budzinski@ens.fr}{thomas.budzinski@ens.fr}.}, \, \textsc{Nicolas Curien}\footnote{Universit\'e Paris-Saclay and Institut Universitaire de France. E-mail: \href{mailto:nicolas.curien@gmail.com}{nicolas.curien@gmail.com}.} \, and \textsc{Bram Petri}\footnote{Sorbonne Universit\'e. E-mail: \href{mailto:bram.petri@imj-prg.fr}{bram.petri@imj-prg.fr}}}

\begin{document}

\maketitle
\begin{abstract} We determine the asymptotic growth rate of the diameter of the random hyperbolic surfaces constructed by Brooks and Makover \cite{BM04}. This model consists of a uniform gluing of $2n$ hyperbolic ideal triangles along their sides followed by a compactification to get a random hyperbolic surface of genus roughly $n/2$. We show that the diameter of those random surfaces is asymptotic to $2 \log n$ in probability as $n \to \infty$.
\end{abstract}

\section{Introduction}

There are several invariants that measure the ``connectedness'' of a closed hyperbolic surface $X$: its diameter $\mathrm{diam}(X)$, its Cheeger constant $h(X)$ and the first non-zero eigenvalue $\lambda_1(X)$ of its Laplacian. The first measures the maximal distance between pairs of points on $X$, the second how hard it is to cut off a large piece from $X$ and the last for instance appears in the rate of mixing of the geodesic flow on $X$. Of course, these three invariants are interrelated: Cheeger \cite{Cheeger} and Buser \cite{BuserNote} proved that lower bounds on $h(X)$ lead to lower bounds on $\lambda_1(X)$ and vice versa, and Brooks \cite{BrooksNumberTheory} proved that a large Cheeger constant implies that the diameter is small.

Looking for the most connected surfaces in the moduli space $\mathcal{M}_g$ of closed hyperbolic surfaces of genus $g$ hence gives rise to three, a priori distinct, optimization problems and leads us to define the following three functions of $g$:
\[ D(g) = \min_{X\in\mathcal{M}_g} \mathrm{diam}(X), \quad H(g) = \sup_{X\in\mathcal{M}_g} h(X) \quad \text{and} \quad \Lambda(g) = \sup_{X\in\mathcal{M}_g} \lambda_1(X).\]
In a companion paper \cite{BCP19b}, we determined the asymptotic behavior of the first when $g$ gets large. Concretely, we proved that $D(g) \sim \log(g)$ as $g\to \infty$. The behavior of $H$ and $\Lambda$ for large genus is less well understood. The best current bounds are
\[\frac{-32+\sqrt{\frac{6923}{2}}}{160}\leq \limsup_{g\to \infty} H(g) \leq 1 \quad \text{and} \quad \frac{975}{4096} \leq \limsup_{g\to\infty}\Lambda(g) \leq \frac{1}{4}. \]
The upper bounds are classical work by Huber \cite{Huber}, Cheng \cite{Cheng} and Cheeger \cite{Cheeger} and the lower bounds come from compactifications of principal congruence covers of $\mathbb{H}^2/\mathrm{PSL}(2,\mathbb{Z})$ and combine work by Kim and Sarnak \cite{Kim}, Brooks \cite{Brooks} and Buser \cite{BuserNote} \footnote{In \cite[Theorem 7.3]{Benson} it is proved that Selberg's conjecture implies that the lower bound on the Cheeger constant can be improved to $0.205\ldots$}. Note that $\frac{-32+\sqrt{\frac{6923}{2}}}{160} \approx 0.168\ldots$ and $\frac{975}{4096} \approx 0.238\ldots$.

% Notes: Kim-Sarnak works for general congruence covers. However, to apply Brooks's result, we need large cusps and systole->inf is an easy way to guarantee this.
% The upper bound on H uses Cheeger's inequality H^2/4 < Lambda - it seems unlikely (see Theorem 5.1 in Brooks-Zuk) that the 1 here is optimal

Another approach to attacking these problems is through random surfaces. Indeed, in recent years, various models of random hyperbolic surfaces have been introduced \cite{BM04,Mir,GPY,BCP19b} and all of these give rise to highly connected surfaces (in all three ways this can be measured). For instance, the result on the diameter in \cite{BCP19b} is based on a random construction. Moreover, Mirzakhani showed that surfaces picked at random in $\mathcal{M}_g$ using the probability measure induced by the Weil-Petersson volume form have Cheeger constant $h\geq \frac{\log(2)}{2\pi+\log(2)} \approx 0.099\ldots $ \cite{Mir}. Finally, it is also expected that some of these models give rise to sequences of closed hyperbolic surfaces whose first eigenvalue converges to $\frac{1}{4}$ (see for instance \cite[Problem 10.3]{Wright}). 

%
%Finally, we note that all three problems also exist in the world of regular graphs, for which more is known. First of all, there are explicit constructions regular of graphs whose first eigenvalue is strictly larger than that of the corresponding regular tree, so called Ramanujan graphs \cite{LPS}. Moreover, it is known that random regular graphs asymptically reach the maximal possible the $\lambda_1$ \cite{Friedman} and the minimal possible diameter \cite{BFdlV}. Somewhat surprisingly, it turns out that Cheeger constants of regular graphs are strictly less than $1$ \cite{Alon}.

\paragraph{The BM model.} In this paper, we investigate the model for random Bely\u{\i} surfaces introduced by Brooks and Makover in \cite{BM04}. This model consists of randomly gluing together $2n$ ideal hyperbolic triangles (with shear $0$) into a complete hyperbolic surface $S^O_n$. This surface is then compactified to obtain a closed hyperbolic surface $S^C_n$ as on Figure \ref{fig_compactification}, see \cite{BM04} and Section \ref{sec:BM} for details. It turns out that the genus $ \mathrm{Genus}(S^C_n)$ of these surfaces is strongly concentrated around $n/2$ as $n\to \infty$ \cite{BM04,Gamburd}.

Besides as a source for highly connected surfaces, the BM model is interesting in its own right. A classical theorem by Bely\u{\i} \cite{Bel} for instance implies that the collection of all the possible surfaces that can be obtained -- all compactifications of shear $0$ gluings of all possible numbers of triangles -- is dense among all hyperbolic surfaces. As such, the BM model is a reasonable model for a ``typical'' hyperbolic surface of large genus (as opposed to the model we employed in \cite{BCP19b}). On top of that, the surfaces sampled according to the BM model show very similar behavior, at least qualitatively, to those sampled using the Weil-Petersson volume form (a phenomenon for which no a priori reason is yet known).

As we mentioned above, Brooks and Makover proved that their surfaces are highly connected and in particular have logarithmic diameter. However, their results are not asymptotically sharp and in fact, the methods they use cannot be expected to yield sharp results. 

The goal of this paper is to determine the asymptotic behavior of the diameter of Brooks and Makover's random surfaces. We prove:

\begin{theorem}\label{thm_main} We have the convergence in probability
\[ \frac{\mathrm{diam}(S^C_n)}{\log n} \xrightarrow[n \to \infty]{(P)} 2.\]
\end{theorem}

What we see is that, perhaps somewhat surprisingly, Brooks and Makover's random surfaces miss the minimal possible diameter by a factor of $2$. This is in stark contrast to the case of regular graphs: Bollob\'as and Fernandez-de la Vega \cite{BFdlV} proved that the diameter of a random trivalent graph on $n$ vertices is concentrated around $ \log_{2} n$, which is the smallest possible diameter for such graphs.

Since a lower bound on the diameter of a surface gives rise to an upper bound on its Cheeger constant and spectral gap, our theorem also gives rise to such bounds. Curiously, because the factor in our theorem is equal to $2$, the bound one obtains is the classical bound 
\[h(S_n^C) \stackrel{(P)}{\leq} 1+o(1) \quad \text{as }n\to\infty.\]
In particular, even if Brooks and Makover's random surfaces are not optimal for the diameter problem, they might still very well be optimizers for the other two problems.

\begin{figure}[!h]
 \begin{center}
 \begin{overpic}[scale=0.95]{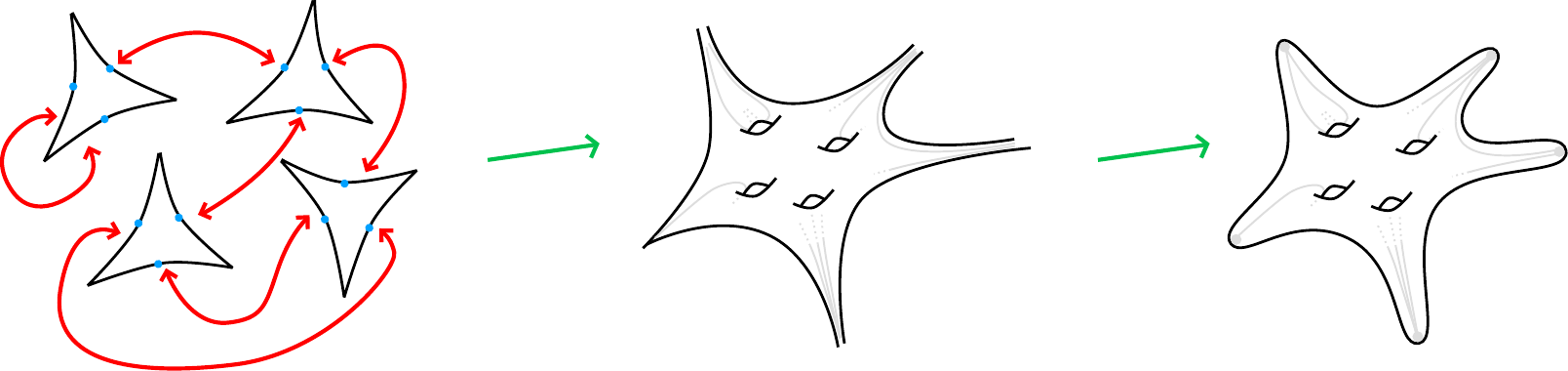}
 \put(33,15.5){(1)}
 \put(72,15.5){(2)}
 \put(57,5){$S^O_n$}
 \put(94.5,5){$S^C_n$}
 \end{overpic}
 \caption{The Brooks-Makover construction of random surfaces. 1) Glue uniformly $2n$ ideal hyperbolic triangles along their sides (with shear $0$ and in an orientable fashion). The resulting random surface $S_{n}^{O}$ is connected with high probability and has approximately $\log n$ cusps corresponding to the vertices of the corresponding triangulation. 2) After putting back those points we get a closed Riemann surface which can be uniformized and yields a random hyperbolic closed surface $S_{n}^{C}$ of genus approximately $n/2$.}\label{fig_compactification}
 \end{center}
 \end{figure}
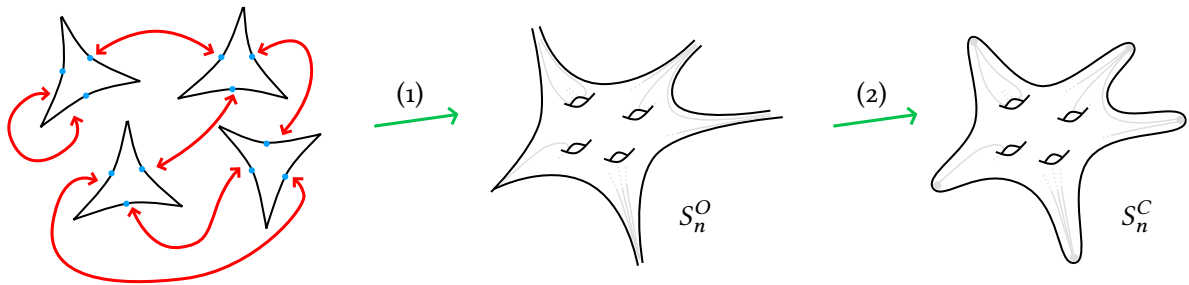
 
\paragraph{A word on the proof.} 
Our proof is a combination of hyperbolic geometry (to control the change in the geometry during the compactification process) and ``peeling'' exploration techniques yielding combinatorial estimates on the triangulation. More precisely, the combinatorics of the dual of $S_{n}^{O}$ is given by a random three-regular graph, the geometry of which is well understood. However, $S_{n}^{O}$ is a hyperbolic surface with roughly $\log n$ cusps (vertices) and, to pass on to $S_{n}^{C}$, we need to compactify it. To understand the effect of the compactification on the geometry, we heavily rely on Brooks's Theorem \cite{Brooks, BM04}, which controls the effect of the compactification sufficiently far from the cusps. Roughly speaking, the cusps of degree $d$ are transformed after compactification into hyperbolic disks of radius $\log d$, and the metric we obtain is made by just identifying the boundaries of all these disks. Using the fact that there exist two cusps with degree proportional to $n$ (see \cite{Gamburd,ChmutovPittel, BCP19a}) with high probability, this already gives the lower bound $  \mathrm{Genus}( S_{n}^{C}) \gtrsim 2 \log n$. For the upper bound, we need to understand how those disks of logarithmic radii are glued back together on the ``bulk'' to form $S_{n}^{C}$. The caricature is that those disks are glued back in a very dense fashion so that many points on their boundaries get close, as sketched on Figure \ref{BrooksSketch}. Making the last sentence rigorous requires to develop quantitative geometric estimates on the random triangulation $ T_{n}$ underlying $S_{n}^{O}$. This is carried out using peeling exploration techniques as developed in \cite{BCP19a}. These estimates (Proposition \ref{combi_estimate_large_vertices}, \ref{combi_estimate_small_vertices} and \ref{combi_estimate_tiny_vertices}) are interesting in their own right since they sharpen our understanding of the geometry of a random triangulation and shed some light on our conjecture in \cite{BCP19a}.

Finally, let us compare the ideas of the proof here with those of \cite{BCP19b}, which also consist of a mixture of probabilistic and geometric arguments. An important difference is that in \cite{BCP19b}, the surface is built from compact pants. Hence, the diameter of such surfaces is essentially the same (up to a constant additive error) as the maximal distance between the centers of the different pairs of pants. In the present work, the building blocks (ideal triangles) are not compact, so this estimate won't directly work. Actually, with the same arguments as in \cite{BCP19b}, the maximal distance between the centers of the ideal triangles can be shown to be asymptotic to $\log(n)$ as $n\to \infty$. So, a posteriori, Theorem \ref{thm_main} also tells us that the diameter is not realized by the centers of the triangles and that the approach from \cite{BCP19b} cannot work for Brooks and Makover's random surfaces.

This also implies that the way we build paths to bound distances is very different here: while the geodesics that realize the diameter in \cite{BCP19b} needed to ``use all the surface'', the paths we will consider here lie mostly in the disks around the vertices of the triangulation, crossing the ``bulk'' (i.e. the yellow part of Figure \ref{BrooksSketch}) only a few times. On the other hand, the peeling explorations used in \cite{BCP19b} and the present paper are of a similar flavor, in the sense that they both try to connect two faces in a ``short'' way for some combinatorial distance.

\begin{figure}[!h]
 \begin{center}
 \begin{overpic}{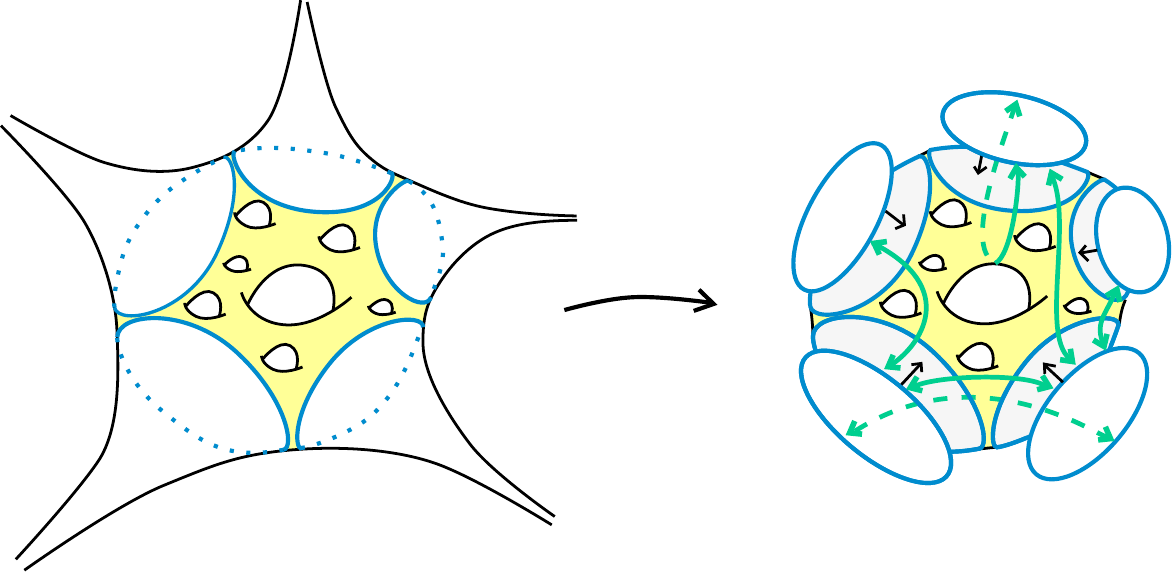}
 \put(43,10){$S^O_n$}
 \put(99,10){$S^C_n$}
 \end{overpic}
 \caption{ \label{BrooksSketch} Caricature of the proof of the main result. Using Brooks's theorem, the surface $S_{n}^{C}$ can roughly be described by gluing hyperbolic disks of radii $\log d_{i}$ onto a dense connected bulk, where $d_{i}$ are the degrees of the vertices, or cusps of $S_{n}^{O}$. In the above figure, $S_{n}^{O}$ has five cusps. Since the bulk creates many connections (in green above) between the boundaries of those disks, the metric is at large scale driven by hyperbolic disks whose boundaries are ``identified''.}
 \end{center}
 \end{figure}

\paragraph{A conjecture on one-vertex triangulations.}
As we can see from the above sketch of proof, the fact that the diameter is $2 \log n$ is mainly due to the presence of several high degree vertices in $S_{n}^{O}$, which after compactification yield well separated points at distance $2 \log n$ from each other. We conjecture that when $S_{n}^{O}$ has a single vertex this phenomemon does not occur. More precisely, let $\tilde{S}_{n}^{O}$ be the random surface obtained by the BM model, conditioned on having a single vertex. By \cite[Appendix B]{Penner} and \cite{BV}, this is an event of probability $ \mathbb{P}( \mathrm{Genus}(S_{n}^{C})=(n+1)/2) \sim \frac{2}{3n}$ as $n \to \infty$.

\begin{conjecture} Let $\tilde{S}_{n}^{C}$ be the compactified version of $ \tilde S_{n}^{O}$. Then $ \displaystyle \frac{\mathrm{diam}(\tilde S^C_n)}{\log n} \xrightarrow[n \to \infty]{(P)} 1$ as $n \to \infty$.
\end{conjecture}

\paragraph{Robustness.}
Finally, a natural question is to to ask whether our result can be extended to models where, instead of building our surface by gluing triangles, we start from another family of polygons as in \cite{BCP19a} \footnote{Note that here, the perimeters of the polygons must all be at least $3$ so that it makes sense to talk about ideal hyperbolic polygons.}. We believe that our arguments should still work with minor adaptations as long as all the polygons have perimeter $n^{o(1)}$. If some faces are larger than that, we expect the result to still be true, but the diameter of one face may become of order $\log n$ after compactification. Hence, it cannot be neglected anymore in the computations.

\subsection*{Acknowledgements}

We thank Maxime Fortier Bourque, Ursula Hamenst\"adt, Fr\'ed\'eric Naud, Hugo Parlier and Juan Souto for useful discussions.

\tableofcontents

\section{Geometric preliminaries}

In this section, we describe the geometry and topology of the random hyperbolic surfaces introduced by Brooks and Makover \cite{BM04}. We will assume some familiarity with the geometry of hyperbolic surfaces. For an introduction, we refer to \cite{BuserBook,Beardon}. We then recall Brooks's theorem which controls the effect of the compactification on the distances in $S_{n}^{C}$ and $S_{n}^{O}$. 

\subsection{Brooks--Makover random surfaces}
\label{sec:BM}

We start by describing the model. For $n \geq 1$, we glue $2n$ oriented ideal hyperbolic triangles along their sides in a uniform fashion. For the gluing along each pair of sides we choose a gluing with shear $0$ and suppose that the gluing respects the orientation of each triangle. This yields with very high probability as $n \to \infty$ a random connected complete hyperbolic surface with $s \approx \log n$ cusps and genus $ \approx \frac{n}{2}$, see e.g. \cite{BCP19a} and the references therein. We shall denote by $T_{n}$ the random triangulation\footnote{To be precise, we could label the edges of the triangles we glued so that $T_{n}$ is in fact a labeled map. Also, in the combinatorics literature, it is generally supposed that a map must be connected. It is not a problem here since $T_{n}$ is connected with probability $1 - O(1/n)$ as $n \to \infty$. See \cite{BCP19a} for details.} describing the combinatorics of $S_{n}^O$, so that the cusps of $S_{n}^{O}$ correspond to the vertices of $T_{n}$, which we denote by $\{v_{1}, v_{2}, ... , v_{s}\}$.

We note that, if we consider $S^O_n$ as a Riemann surface, the cusps have neighborhoods that are biholomorphic to punctured disks in the complex plane. As such, these surfaces have a natural compactification: the Riemann surface $S^C_n$ obtained by adding the points back in. The uniformization theorem now supplies us with a unique Riemannian metric of constant curvature $-1$, $0$ or $1$ on $S^C_n$. We call $S^C_n$ the \emph{conformal compactification} of $S^O_n$. Since the genus of $S^C_n$ is larger than $2$ with high probability, this metric is typically hyperbolic and will be denoted by $d_{\hyp}$. In the rest of the paper, we shall always identify $S_{n}^{O}$ and its cusps $\{v_{1}, ... , v_{s} \}$ with $S_{n}^{C}$ as point sets, but it should be clear from the context which metric we consider (either the metric $ds^2_{S^O}$ on $S_{n}^{O}$, or the metric $ds^2_{S^C}$ on $S_{n}^{C}$, or some combinatorial information about the random triangulation $T_{n}$).

\subsection{The change in geometry after compactification}

The main geometric aspect of the surfaces $S^C_n$ that we need to control is how they look near the points we added in the cusps of $S^O_n$. We will discuss this in this section. 

The main conclusion, that will use results by Brooks and Brooks-Makover, will be that up to a bounded error the metric $d_{\hyp}$ can be described as follows. First we pick horocycles of some large but fixed length $L>0$ around all the cusps in $S^O_n$. It turns out that for $n$ large enough, such horocycles typically determine disjoint neighborhoods of the cusps of $S^O_n$. Hence, we can remove these cusp neighborhoods and replace them with hyperbolic disks of perimeter $L$ (like in Figure \ref{BrooksSketch}). This gives us a closed surface homeomorphic to $S^C_n$ with a metric on it. Of course, this metric is not quite hyperbolic, and the disks do not glue very nicely on the bits of $S^O_n$. However, as we will argue in this section, this is a reasonable model for the geometry of the hyperbolic metric on $S^C_n$.

In order to formalize this description, we will use a theorem by Brooks. First, we need a definition.
\begin{definition}\label{def_cusplength}
Let $L>0$ and let $S$ be a hyperbolic surface with cusps $ v_1,\ldots,v_s$. We say that $S$ has \emph{cusp length} $\geq L$ if there exist horocycles $h_1,\ldots,h_s$ such that
\begin{itemize}
\item $h_i$ is a horocycle around $v_i$ of length larger than or equal to $L$  for all $i$,
\item $h_i$ is homeomorphic to a circle,
\item $h_i\cap h_j = \emptyset$ for all $i\neq j$.
\end{itemize}
\end{definition}

Brooks' theorem, which is an entirely deterministic result, is now as follows:

\begin{theorem}\label{thm_brooks}\cite[Theorem 2.1]{Brooks} For every $\varepsilon >0$, there exists an $L=L(\varepsilon) > 0$ such that the following holds. Let $S^O$ be a hyperbolic surface with cusps $\{v_{1}, ... , v_{s} \}$ that has cusp length $\geq L$ and denote by $S^C$ its conformal compactification. We identify $S^{O} \cup \{v_{1}, ..., v_{s}\}$ and $S^{C}$ as point sets and write $B_i(R)$ for the disk of radius $R$ around the point $v_{i}$ for the metric in $S^{C}$ and $N_i(L)$ for the cusp neighborhood defined by a horocycle of length $L$ around the cusp $v_{i}$ in $S^{O}$.  
\begin{itemize}
\item[1.] For all $i=1,\ldots,s$: 
\[B_i ({(1+\varepsilon)^{-3}R}) \subseteq N_i(L) \cup \{v_{i}\} \subseteq B_{ i}(R), \]
where $  \displaystyle R=(1+\varepsilon)^{3/2}\log\left(\frac{e^{2\pi/L}+1}{e^{2\pi/L}-1}\right)$. %\underset{L \to \infty}{\sim} (1+\varepsilon)^{3/2} \log\left(L \right).$
\item [2.] Outside $\cup_i B_i(R)$  we have 
          \begin{eqnarray}\frac{1}{1+\varepsilon} ds^2_{S^O} \leq ds^2_{S^C} \leq  (1+\varepsilon)ds^2_{S^O}.   \label{eq:compare}\end{eqnarray}
\end{itemize}
\end{theorem}

So in order to control the geometry of the compactified surface, we need to find ``large'' horocycles. To this end, we once and for all fix $\varepsilon>0$ and $L=L(\varepsilon)$ given by Theorem \ref{thm_brooks}. In \cite[Theorem 2.1]{BM04}, Brooks and Makover prove that
\begin{equation}\label{eq:cusp_length}
\PP(\text{the surface }S^O_n\text{ has cusp length }\geq L) \xrightarrow[n \to \infty]{} 1.
\end{equation}

Let us look at the argument in \cite{BM04}. In big lines, this runs as follows. First of all, one can draw horosegments of length $1$ on all the triangles that match up into horocycles around the cusps that don't intersect each other. The resulting horocycle around a cusp with $d$ triangles around it has length $d$. This means that the only problem are short cycles in the dual graph of the triangulation. The solution is to push the resulting horocycles out in order to make them longer. Of course, in order to ensure that the collection of horocycles remains disjoint, the other horocycles need to be shrunk. The reason why this can be done is that short cycles are far away from each other in the dual graph (a result due to Bollob\'as \cite{Bollobas}).

Arbitrarily labeling the cusps of $S^O_n$ by $v_1,\dots, v_s$, we denote by $h^{(1)}_j$ the horocycles of length $L$ around the cusps $v_j$ given by \cite{BM04}. The horocycle neighborhoods they determine will be denoted $N^{(1)}_j \equiv N_{j}(L)$. The fact that these horocycles are disjointly embedded allows us to apply  Theorem \ref{thm_brooks}.

\subsection{Canonical horocycles}
The above result enables us to understand, with high probability, the geometry of $S_{n}^{C}$ by replacing the horocycle neighborhoods $N_{j}^{(1)}$ with hyperbolic disks of perimeter $L$, hence radius roughly $\log L$. In the case of a cusp of large degree $d >> L$, this control is not sufficient: recall from the introduction that we want to replace a horocycle neighborhood with a disk of radius roughly $\log d$. To do this, we shall consider a second set of horocycles $h^{(2)}_j$ around the cusps $v_j$ that determine horocycle neighborhoods $N^{(2)}_j$ around $v_{j}$ for  $j=1,\ldots, s$. We build these out of all the horocycle segments of some fixed length $\alpha=\alpha(\varepsilon)$ in the ideal triangles we started with. The value of $\alpha$ will be specified a few lines below. Note that this means that in $S_{n}^{O}$
\[  \mathrm{Length}(h^{(2)}_j) = \alpha \cdot d_j,\]
where $d_j$ denotes the degree of the cusp $v_j$ (i.e. the degree of the corresponding vertex in $T_{n}$).

The choice of $\alpha$ will be constrained by the following two conditions, that we want our horocycles $h_j^{(2)}$ to satisfy:
\begin{enumerate}
\item they have to be disjoint;
\item if $i\neq j$ then we want $N_j^{(2)}\cap N_i^{(1)} = \emptyset$.
\end{enumerate}

Both of these are conditions on $\alpha$. The first of these conditions is satisfied as long as $\alpha \leq 1$. In order to guarantee the second condition, we need to make the process of shrinking horocycles described right after \eqref{eq:cusp_length} somewhat quantitative.

The worst case for the process described above is a loop in the dual graph, i.e. an ideal triangle of which two sides are identified in the gluing. Figure \ref{pic_horocycles2} shows a picture of an ideal triangle with vertices $0$, $1$ and $\infty$ in the hyperbolic plane, where we try to build a horocycle around the cusp at $\infty$.

\begin{figure}[h]
  \begin{center}
  \begin{overpic}{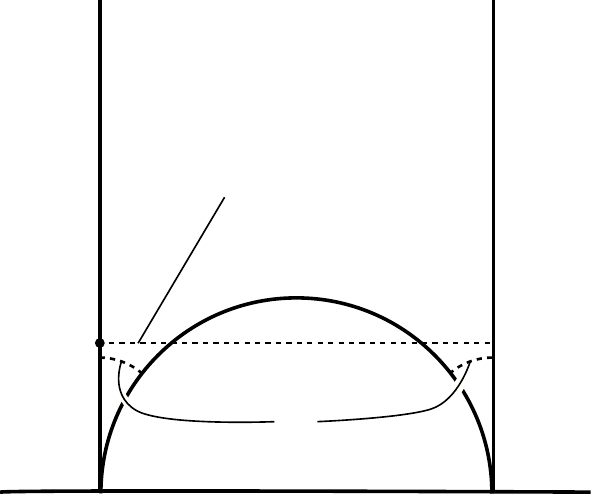}
   \put (4,24) {$\frac{i}{L}$}
   \put (37,52) {$L$}
   \put (49,10) {$\alpha$}
   \put (16,-6) {$0$}
   \put (82,-6) {$1$}
  \end{overpic}
  \caption{Horocycles around the cusp consisting of a single triangle in the half-plane model. The vertical left and right boundaries are identified in $S_{n}^{O}$.}\label{fig_horocycle_one_triangle}
  \label{pic_horocycles2}
  \end{center}  
\end{figure}
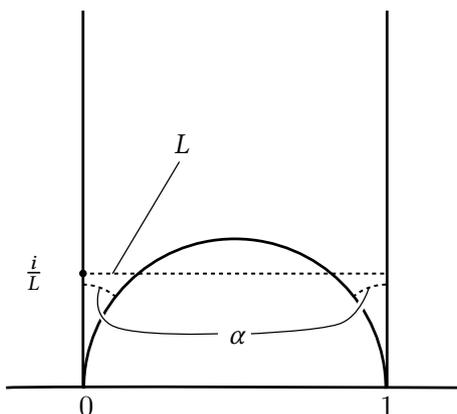

In order to obtain a horocycle of length $L$ around such a cusp, we need to use the projection of the horosegment at height $1/L$. So $\alpha=\alpha(\varepsilon)\in (0,1)$ is chosen so that the two horosegments of length $\alpha$ around the other two vertices of our triangle are disjoint from (and below) the horosegment at height $1/L$. Note that this means that $\alpha$ depends on $L$ and hence $\varepsilon$ only. Since the geometry of $N_j^{(2)}$ is entirely determined by how many triangles are incident to the corresponding cusp, this allows us to translate combinatorial properties of the triangulation into geometric properties of the "combinatorial model" of Figure \ref{BrooksSketch}, and hence of the surface $S^O_n$.

\subsection{Rough geometric estimates}
\label{sec:geo}
Let us derive some rough geometric estimates on the hyperbolic metric $ d_{ \mathrm{hyp}} $ on $S_{n}^{C}$ using the above constructions. Recall that $ \varepsilon>0$ is fixed, $L=L( \varepsilon)$ and $R=R( \varepsilon,L)$ are given in Theorem \ref{thm_brooks}, and $\alpha= \alpha( \varepsilon)$ is given in the last subsection. We suppose that $S_{n}^{O}$ is connected, has genus larger than $2$, has cusps $v_{1}, ... , v_{s}$ and has cusp length larger than $L$ (all of this happens with high probability as $n \to \infty$). We recall the notation $N^{(1)}_{j} = N_{j}(L)$ and $N^{(2)}_{j}=N_{j}( \alpha \cdot d_{j})$ for the cusp neighborhoods around $v_{j}$ of length $L$ and $ \alpha \cdot d_{j}$ respectively. The setup is summarized in  Figure \ref{fig:horos}.

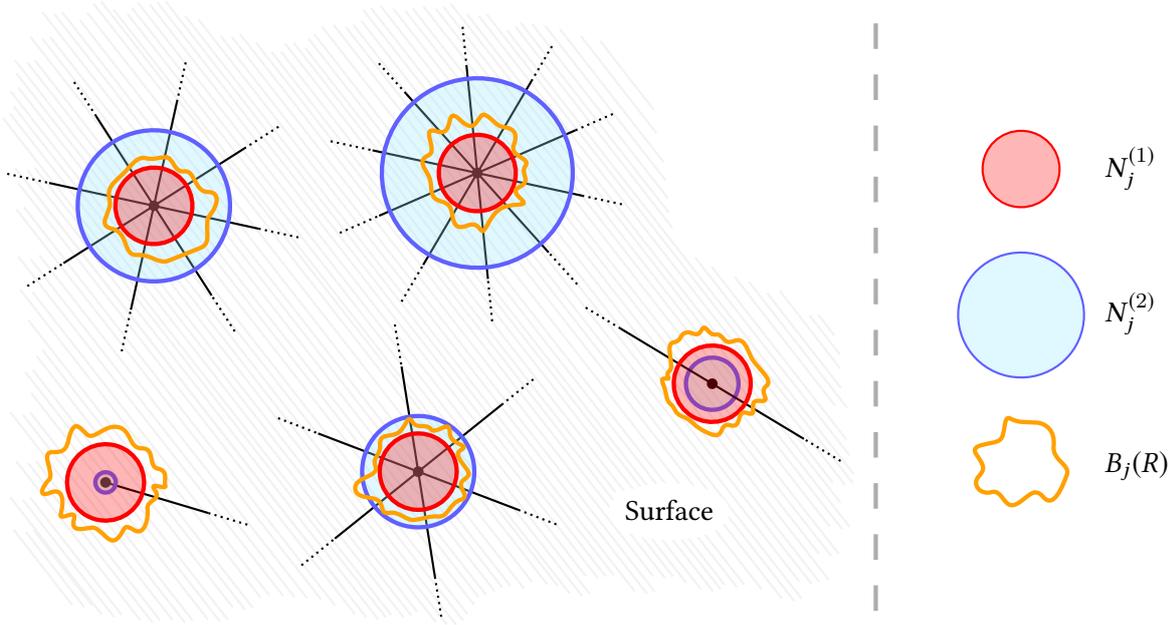
\begin{figure}[!h]
 \begin{center}
 \begin{overpic}{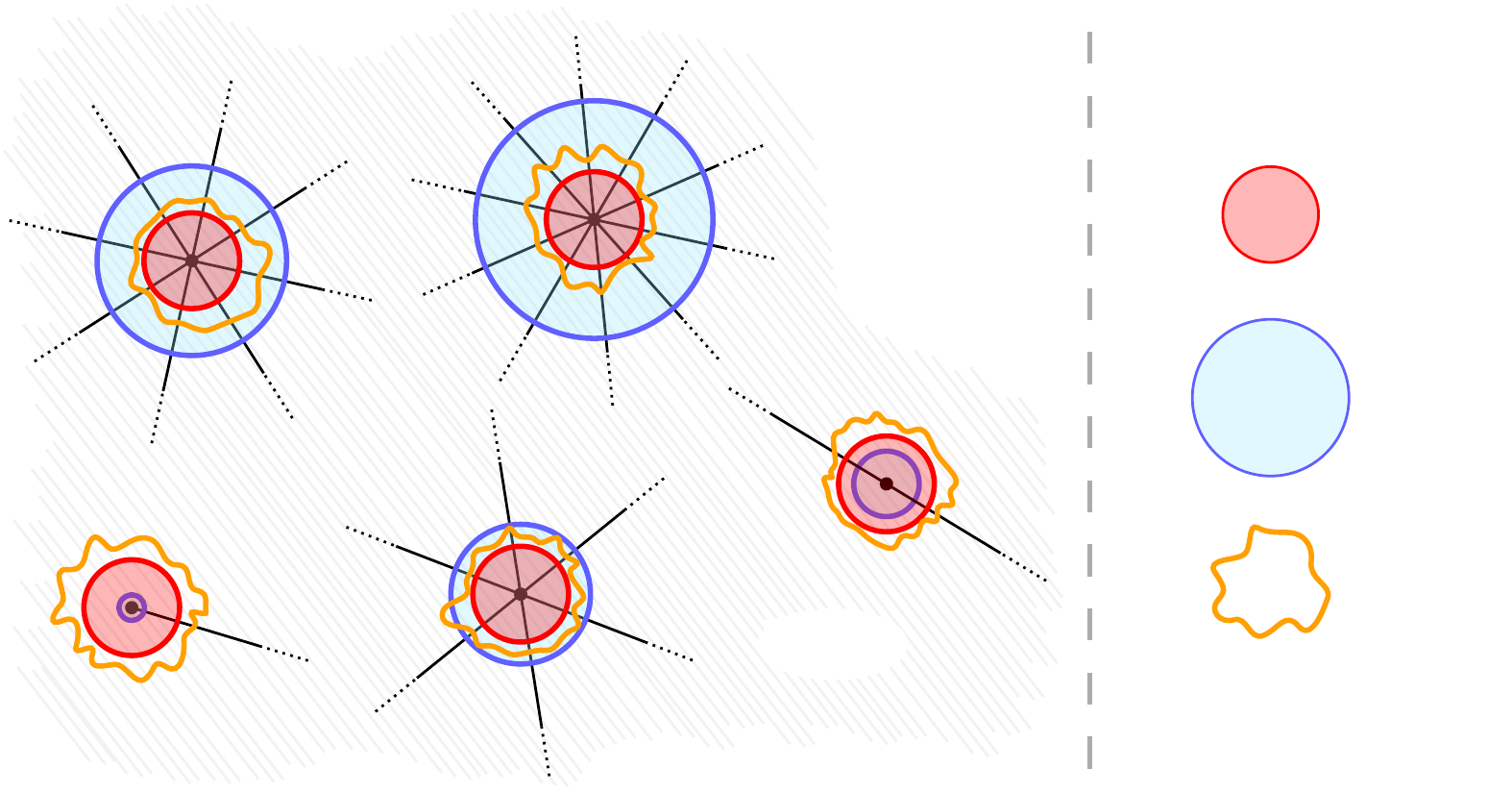}
 \put(51.5,9){Surface}
 \put(91,37.5){$N^{(1)}_j$}
 \put(91,25.5){$N^{(2)}_j$}
 \put(91,13){$B_j(R)$}
 \end{overpic}
 \caption{ \label{fig:horos}Setup of our geometric estimates here for a surface with $5$ cusps of degrees $8,10,2,6$ and $1$. The red horocycles have length $L$, they exist with probability tending to $1$ as $n \to \infty$ thanks to \cite{BM04} and enable us to apply Theorem \ref{thm_brooks}. Their neighborhoods are contained in the orange balls $B_{j}(R)$ for the metric on $S_{n}^{C}$. The blue horocycles have length $ \alpha \cdot d_{j}$, where $\alpha>0$ is chosen as before. For small degrees those cusp neighborhoods are smaller than the red ones, and vice versa for large degrees.  }
 \end{center}
 \end{figure} 

Let us make a few geometric remarks about the \emph{metric $ d_{ \mathrm{hyp}}$ in $S_{n}^{C}$} as $n \to \infty$:
\begin{enumerate}
\item each of the orange regions $B_{j}(R)$ has diameter bounded by $2R$,
\item any $x \in S_{n}^{C}$ is within bounded distance of the union $ \cup N_{j}^{(2)}$ of the blue cusp neighborhoods. Indeed this is true for the metric in $S^{O}_{n}$, and since by \eqref{eq:compare} those two metrics are comparable outside $\cup B_{j}(R)$, the statement follows from the last point.
\item \textbf{Lower bound.} For $i \ne j$, the distance in $S_n^C$ between $v_{i}$ and $v_{j}$ satisfies 
 \begin{eqnarray} \label{eq:lb} {d}_{ \mathrm{hyp}}(v_{i}, v_{j}) \geq (1- \varepsilon) \left(\log ( \alpha \cdot d_{i}) + \log ( \alpha \cdot d_{j})\right) - 4R.  \end{eqnarray}
Indeed,  \eqref{eq:compare} and the first point of the list enable us to see, up to a multiplicative error $(1 \pm \varepsilon)$ and an additive error $ \pm 2R$, each blue region $N^{(2)}_{j}$ as a hyperbolic disk of radius $\log (\alpha \cdot d_{j})$. The distance between their centers is then bounded from below by the quantity in the last display.
\item \textbf{Upper bounds.} We shall also need an estimate on the distance between two points $x_{1},x_{2} \in S_{n}^{C}$ which lie in a same blue neighborhood $N^{(2)}_{j}$. Let us zoom in on such a neighborhood (see Figure \ref{fig:geocusp} with the same drawing convention as in Figure \ref{fig:horos}).
\begin{figure}[!h]
 \begin{center}
 \begin{overpic}[width=7cm]{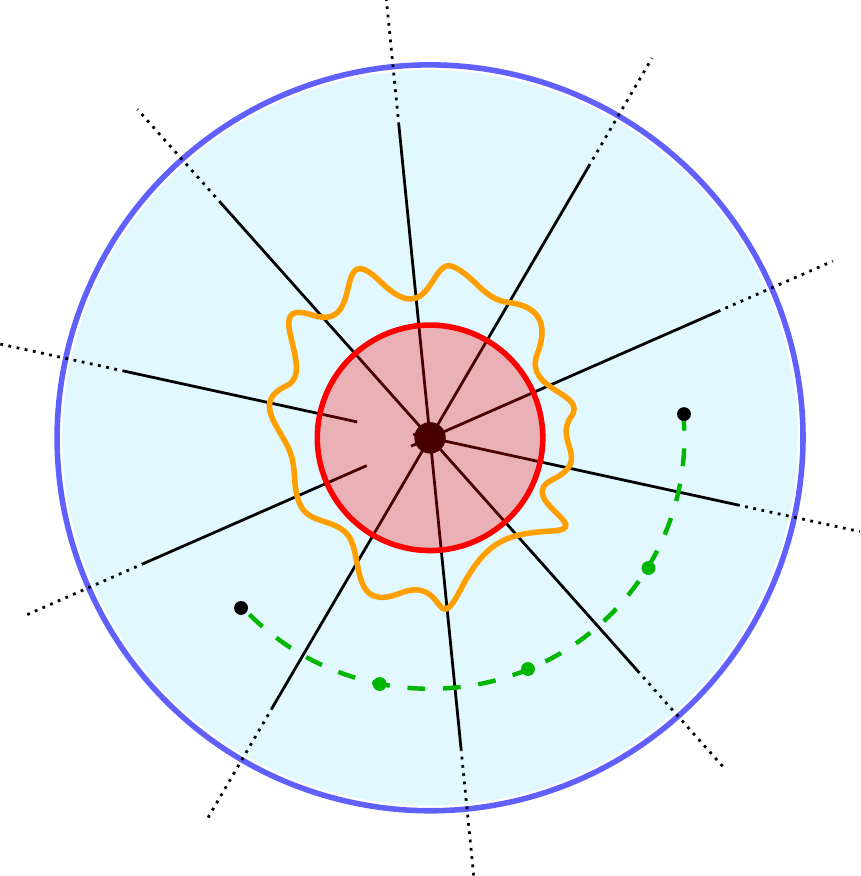}
 \put(24,26){$x_1$}
 \put(80,51){$x_2$}
 \put(42,48.5){$v_j$}
 \end{overpic}
 \caption{ \label{fig:geocusp} Two points $x_{1},x_{2}$ lying in the same cusp neighborhood $N_{j}^{(2)}$. One can evaluate their distances within this region using their distances to the horocycle and the combinatorial dual distance between their corners (which is $4$ in this example).}
 \end{center}
 \end{figure}
 If $x_{1},x_{2}$ belong to a common cusp neighborhood $N_{j}^{(2)}$, one can define the combinatorial dual distance between $x_{1}$ and $x_{2}$ as  the number of half-edges of $T_{n}$ that one needs to cross when moving around $v_{j}$ from $x_1$ to $x_2$ in the shortest direction. This distance only depends on the corners $c_{1},c_{2}$ associated to  $x_{1},x_{2}$ and we shall denote it by $d_{v_{j}}(c_{1},c_{2})$. Note that once the horocycles $h^{(2)}_j$ have been defined, the corners can be canonically defined as the "triangles" delimited by two sides of a face of $T_{n}$ and a portion of horocycle. Recalling from above that up to a bounded additive error $\pm 2R$ and up to a multiplicative error $1 \pm \varepsilon$, the region $N^{(2)}_{j}$ can be seen as a hyperbolic disk of radius $\log ( \alpha \cdot d_{j})$. Furthermore, the boundaries of the triangles can be seen as radii of this disk such that two consecutive of them make an angle $\frac{2\pi}{d_j}$. Therefore, the angle between the three points $x_1$, $v_j$ and $x_2$ is close to $\frac{2\pi}{d_j} d_{v_j}(c_1,c_2)$. Using elementary geometry (more precisely, the hyperbolic cosine law), we deduce that if $d_{\hyp} (x_1,h_j^{(2)}) = d_{\hyp} (x_2,h_j^{(2)})$, then\footnote{We could estimate this distance for any two points in $N_j^{(2)}$, but we will not need the general case, and the formula is simpler if $d_{\hyp} (x_1,h_j^{(2)}) = d_{\hyp} (x_2,h_j^{(2)})$.}
 \begin{eqnarray} d_{ \mathrm{hyp}}(x_{1},x_{2}) &\leq& 2(1+\varepsilon) \; \max\Big\{ \log(d_{v_j}(c_1,c_2)) - d_{\hyp} (x_1,h_j^{(2)}), \; 0   \Big\} + 2R.   \label{eq:ub} \end{eqnarray}
Finally, let $x_1$ and $x_2$ belong to the same face $f$, but to two different cusp neighbourhoods $N_i^{(2)}$ and $N_j^{(2)}$. Then any two points on $h_i^{(2)} \cap f$ and $h_j^{(2)} \cap f$ lie at bounded distance (by a constant $C(\eps)$) from each other. Hence, we have
\begin{equation}\label{eq:ub2}
d_{\hyp}(x_1,x_2) \leq (1+\eps) \left( d_{\hyp} (x_1,h_i^{(2)}) + d_{\hyp} (x_2,h_j^{(2)})\right) + 4R + C.
\end{equation}
\end{enumerate}

\section{Proof of the theorem given combinatorial estimates}
Let us now prove our main theorem relying on some combinatorial estimates on the random triangulation $T_{n}$ that will be proved in the last section. We divide the proof into lower and upper bound.
\subsection{Lower bound}
\begin{proof}[Proof of the lower bound of Theorem \ref{thm_main}] The crucial observation is that with high probability, there are at least two vertices of degree proportional to $n$. More formally, if $D_{n}(1), D_{n}(2), ... $ are the vertex degrees ranked in decreasing order, then by \cite{Gamburd} (see also \cite{ChmutovPittel,BCP19a}), we have the convergence in distribution $$\left(  \frac{D_{n}(1)}{6n}, \frac{D_{n}(2)}{6n}, ... \right) \xrightarrow[n\to\infty]{(d)}  \mathbf{PD},$$ where $\mathbf{PD}$ is the Poisson--Dirichlet distribution with values in the infinite simplex $\{x_{1} \geq x_{2} \geq \cdots > 0 :  \sum_i x_i = 1\}$. Therefore, for every $\eps>0$, there is $\delta>0$ such that, for $n$ large enough, we have
\[ \PP \left( D_n(2) \geq \delta n \right) \geq 1-\eps.\]
Reasoning on the above event intersected with the conditions imposed in the beginning of Section \ref{sec:geo}, we deduce thanks to \eqref{eq:lb} that on this event (of asymptotic probability larger than $1 - \varepsilon$) we have as desired 
$$ \mathrm{Diam}( S_{n}^{C}) \geq  2(1- \varepsilon) \log ( \alpha \cdot  \delta n) - 4R(\varepsilon) \underset{n \to \infty}{\sim} 2 (1- \varepsilon) \log n.$$ \end{proof}
%and the lower bound follows since $\frac{ \mathrm{Genus}(S_n^C)}{n}$ converges to $\frac{1}{2}$ in probability as $n \to \infty$.

\subsection{Upper bound}
For the upper bound, we shall use combinatorial estimates on the random triangulation $T_{n}$ which will be proved in the Section \ref{sec:proof_combi}.
\paragraph{Combinatorial estimates} 
\begin{proposition}\label{combi_estimate_large_vertices}

Let $\eps \in (0,1)$. With high probability as $n \to \infty$, the following holds. For any two corners $c_1, c_2$ of $T_n$ of two faces $f_1, f_2$ incident to two vertices $v_1, v_2$ such that $\deg(v_1) \deg(v_2) \geq n^{1+\eps}$, there is a face $f'$ incident both to $v_1$ at a corner $c'_1$ and to $v_2$ at a corner $c'_2$, and such that
\[d_{v_1}(c_1,c'_1) \leq 3n^{\beta_1/(\beta_1+\beta_2-\eps/2)} \quad \mbox{ and } \quad d_{v_2}(c_2,c'_2) \leq 3n^{\beta_2/(\beta_1+\beta_2-\eps/2)},\]
where $\beta_i$ is such that $\mathrm{deg}(v_{i}) = n^{\beta_{i}}$  for $i \in \{1,2\}$.
\end{proposition}
\begin{figure}[h]
  \begin{center}
  \begin{overpic}{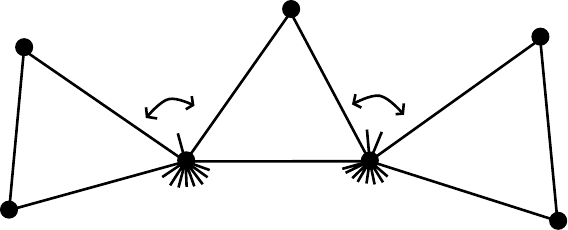}
   \put (8,15) {$f_1$}
   \put (23,12) {$c_1$}
   \put (37,15) {$c'_1$}
   \put (70,11) {$c_2$}
   \put (55,15) {$c'_2$}
   \put (30,2) {$v_1$}
   \put (85,14) {$f_2$}
   \put (62,2) {$v_2$}
   \put (47,23) {$f'$}
   \put (14,28) {$d_{v_1}(c_1,c'_1)$}
   \put (57.5,28.5) {$d_{v_2}(c_2,c'_2)$}       
  \end{overpic}
  \caption{Proposition \ref{combi_estimate_large_vertices}: two corners incident to ``large'' vertices share a touching face.}
  \label{pic_largevertices}
  \end{center}  
\end{figure}
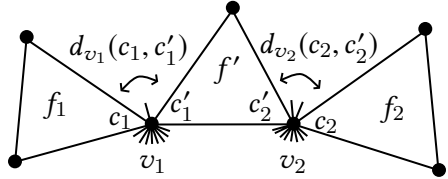

\begin{proposition}\label{combi_estimate_small_vertices}
Let $\eps \in (0,1)$. With high probability as $n \to \infty$, the following holds.  For any two vertices $v_1$, $v_2$ of $T_n$ such that $\deg(v_1) \deg(v_2) \leq n^{1+\eps}$ but $\deg(v_1), \deg(v_2) \geq n^{2\eps}$, one of the two following assertions hold:
\begin{enumerate}
\item
there is a face $f''$ incident to both $v_1$ and $v_2$,
\item
there are a vertex $v'$ and two faces $f'_1$ and $f'_2$ such that:
\begin{itemize}
\item
$f'_1$ is incident to both $v_1$ and $v'$ (at a corner $c'_1$);
\item
$f'_2$ is incident to both $v_2$ and $v'$ (at a corner $c'_2$);
\item
$d_{v'}(c'_1, c'_2) \leq \frac{n^{1+2\eps}}{\deg(v_1) \deg(v_2)}$.
\end{itemize}
\end{enumerate}
\end{proposition}

\begin{figure}[h]
  \begin{center}
  \begin{overpic}{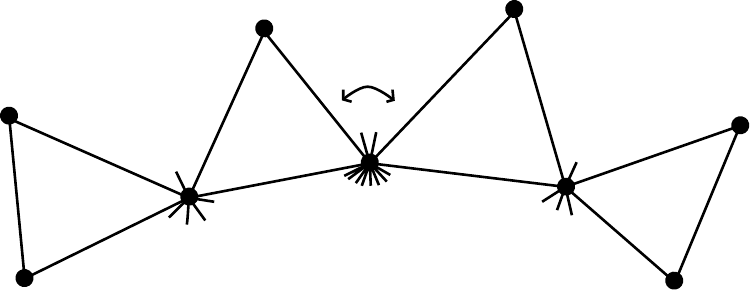}
   \put (8,12) {$f_1$}
   \put (41,18) {$c'_1$}
   \put (53,18) {$c'_2$}
   \put (24,5.5) {$v_1$}    
   \put (32,20) {$f_1'$}
   \put (85,11) {$f_2$}
   \put (73,6.5) {$v_2$}
   \put (63,20) {$f_2'$}
   \put (47,9) {$v'$}
   \put (40,32) {$d_{v'}(c'_1,c'_2)$}     
  \end{overpic}
  \caption{Illustration of the second option in Proposition \ref{combi_estimate_small_vertices}.}
  \label{pic_smallvertices}
  \end{center}  
\end{figure}
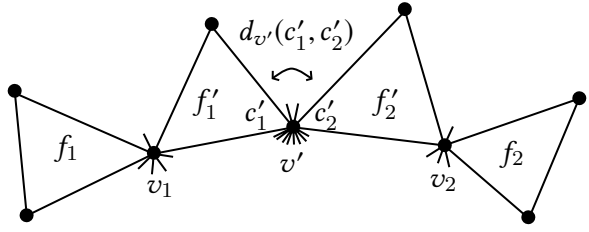

\begin{proposition}\label{combi_estimate_tiny_vertices}
With high probability as $n \to \infty$, every vertex of $T_n$ is at graph distance at most $6$ from a vertex of degree at least $n^{1/4}$.
\end{proposition}

With those estimates at hands, we can proceed to the proof of the upper bound of our main result.

\begin{proof}[Proof of the upper bound of Theorem \ref{thm_main}] Recall that $ \varepsilon>0$ is fixed and suppose $ \varepsilon \in (0, 1/8)$. To ease notation, we will write $f_{n} \preceq g_{n}$ if eventually $f_{n} \leq g_{n} + A$, where $A=A( \varepsilon) >0$ whose value may increase from line to line but only depends on $ \varepsilon>0$ (in particular $A$ is not random). Our goal is to prove that on an event $ \mathcal{E}_{n}$ such that $  \mathbb{P}(\mathcal{E}_n) \to 1$, we have $ \mathrm{diam}(S_{n}^{C}) \preceq 2 (1 + \varepsilon) \log n$. %Since $ \mathrm{Genus}(S_{n}^{C}) / n \to 1/2$ in probability this will finish the proof of the theorem. 
Our event $ \mathcal{E}_{n}$ will be the intersection of the event on which the geometric conclusions of Section \ref{sec:geo} hold true, together with the events on which the conclusions of the above Propositions \ref{combi_estimate_large_vertices}, \ref{combi_estimate_small_vertices} and \ref{combi_estimate_tiny_vertices} hold true. From now on, we argue on this event and the rest of the reasoning is deterministic.\medskip

Let $x_{1},x_{2} \in S_{n}^{C}$. By the first item of the list in Section \ref{sec:geo}, up to loosing an additive constant, one can suppose that $x_{1}$ and $x_{2}$ are in some blue cusp neighborhoods, say $N_{i}^{(2)}$ and $N_{j}^{(2)}$, whose associated cusps have degrees $d_{i}, d_{j} \geq 1$. If $i=j$ then $ d_{ \mathrm{hyp}}(x_{1},x_{2}) \preceq 2 (1+ \varepsilon) \log n$ by our geometric considerations (the neighbourhood $N_i^{(2)}$ is close to a ball of radius $\log d_i \leq \log n$). We thus focus on the case $i \ne j$ and suppose $i=1$ and $j=2$ to fix notation. We will bound the distance $d_{\hyp}(x_1,x_2)$ in three different cases according to the values of the degrees $d_1$ and $d_2$. Each of these cases corresponds to one of the Propositions \ref{combi_estimate_large_vertices}, \ref{combi_estimate_small_vertices} and \ref{combi_estimate_tiny_vertices}.

\noindent\textbf{Case 1:} $d_1 d_2 \geq n^{1+\eps} $.
In this case, we use Proposition \ref{combi_estimate_large_vertices} and link $x_{1}$ to $x_{2}$ as follows. Let $f', c'_1, c'_2$ be the face and the two corners given by Proposition \ref{combi_estimate_large_vertices} for the corners $c_1$ and $c_2$ associated with $x_{1}$ and $x_{2}$. Now let $x_1'$ and $x_2'$ be two points respectively in the corners $c'_1$ and $c'_2$ of $f'$, such that $d_{\hyp} (x_1,h_1^{(2)}) = d_{\hyp} (x_1',h_1^{(2)})$ and $d_{\hyp} (x_2,h_2^{(2)}) = d_{\hyp} (x_2',h_2^{(2)})$. Then Eq.~\eqref{eq:ub} tells us that
\[ d_{\hyp}(x_1,x_1') \preceq 2(1+\varepsilon) \; \Big( \log(d_{v_1}(c_1,c'_1)) - d_{\hyp} (x_1,h_1^{(2)})  \Big)\]
and
\[ d_{\hyp}(x_2,x_2') \preceq 2(1+\varepsilon) \; \Big( \log(d_{v_2}(c_2,c'_2)) - d_{\hyp} (x_2,h_2^{(2)})  \Big), \]
and Eq.~\eqref{eq:ub2} gives
\[d_{\hyp}(x'_1,x'_2) \preceq (1+\eps) \left( d_{\hyp} (x_1,h_1^{(2)})+d_{\hyp} (x_2,h_2^{(2)}) \right).\]
Adding these up and using the bounds given by Proposition \ref{combi_estimate_large_vertices}, we obtain
\begin{align*}
d_{\hyp} (x_1,x_2) & \leq d_{\hyp}(x_1,x'_1)+d_{\hyp}(x'_1,x'_2)+d_{\hyp}(x'_2,x_2)\\
& \preceq 2(1+\varepsilon) \left( \log d_{v_1}(c_1,c'_1) + \log d_{v_2}(c_2,c'_2) \right) - (1+\varepsilon) \left( d_{\hyp} (x_1,h_1^{(2)}) + d_{\hyp} (x_2,h_2^{(2)}) \right)\\
& \preceq  2(1+2\varepsilon) \log(n).
\end{align*}

\noindent\textbf{Case 2:} $d_1 d_2 \leq n^{1+\eps} $ and $d_1,d_2 \geq n^{2\varepsilon}$. Now we will use Proposition \ref{combi_estimate_small_vertices}. In the first case (if we have a face $f''$ incident to both $v_1$ and $v_2$), by the geometric considerations gathered in Section \ref{sec:geo} we get
\begin{align*}
d_{\hyp}(x_1,x_2) & \leq d_{\hyp}(x_1,v_1)+d_{\hyp}(v_1,v_2)+d_{\hyp}(v_2,x_2)\\
& \preceq 2 (1+\eps) \left( \log d_1 + \log d_2 \right)\\
& \preceq 2(1+\eps)^2 \log n.
\end{align*}
In the second case of Proposition \ref{combi_estimate_small_vertices}, let $v'$, $f'_1$ and $f'_2$ be the vertex and faces given by Proposition \ref{combi_estimate_small_vertices}. The vertex $v'$ is incident to the corner $c'_1$ of $f'_1$ and to the corner $c'_2$ of $f_2$. Let $x'_1 \in f'_1 \cap \partial N_{v'}^{(2)}$ and $x'_2 \in f'_2 \cap \partial N_{v'}^{(2)}$. By Eq.~\eqref{eq:ub} and \eqref{eq:ub2}, we have
\[d_{\hyp}(x_1,x_1') \preceq 2(1+\varepsilon) \log(d_1) \quad \mbox{ and } \quad d_{\hyp}(x_2,x_2') \preceq 2(1+\varepsilon) \log(d_2).\]
Moreover, by \eqref{eq:ub} and the bound given by Proposition \ref{combi_estimate_small_vertices}, we also have
\[d_{\hyp}(x_1',x_2') \preceq 2(1+\varepsilon) \log d_{v'}(c'_1,c'_2)  \preceq 2(1+\varepsilon) \log\left( \frac{n^{1+2\varepsilon}}{d_1 d_2} \right) .\]
Adding everything up, we obtain as desired
\[d_{\hyp}(x_1,x_2)  \preceq 2(1+\eps)(1+2\varepsilon) \; \log(n). \]

\noindent\textbf{Case 3:} $d_1\leq n^{2\varepsilon}$ or $d_2 \leq n^{2\varepsilon}$. Assume $d_1 \leq n^{2\varepsilon}$. Let $v'_1$ be the closest vertex from $v_1$ (for the graph distance in $T_n$) with degree at least $n^{2\eps}$. Since $2\eps<1/4$, by Proposition \ref{combi_estimate_tiny_vertices}, there is a path with graph length at most $6$ from $v_1$ to $v'_1$ using only vertices with degrees at most $n^{2\eps}$ (except of course $v'_1$). But by Eq. \eqref{eq:ub2}, the hyperbolic distance $d_{\hyp}$ between two neighbour vertices of $T_n$ of degree $\leq n^{2\varepsilon}$ is at most
\[2 \varepsilon(1+\varepsilon) \log(n) + 4R.\]
Therefore, up to paying roughly $2 \times 6 \times 2 \varepsilon(1+ \varepsilon) \log n$, we can replace $x_1$ by a point $x_{1}'$ in a cusp neighbourhood of degree larger than $n^{2\eps}$. The same is true for $x_2$ if $d_2 \leq n^{2\eps}$, so we are back to case $1$ or $2$. By finally letting $\eps \to 0$, this concludes the proof of the theorem.
\end{proof}

\section{Proof of the combinatorial estimates}\label{sec:proof_combi}
Our goal is now to prove Propositions \ref{combi_estimate_large_vertices}, \ref{combi_estimate_small_vertices} and \ref{combi_estimate_tiny_vertices}. These results only deal with the random triangulation $T_{n}$ which is built by gluing $2n$ triangles in a uniform fashion. Our main tools will be exploration methods of such maps as in \cite{BCP19a}. Those estimates are interesting on their own, since they sharpen our understanding of the geometry of the graph structure of $T_{n}$ and give further support for \cite[Conjecture 1]{BCP19a}. 

\subsection{Peeling explorations of random triangulations}\label{subsec_defn_peeling}
We recall some background from \cite{BCP19a}, which treats a more general setting. We fix $n \geq 1$, and a pairing $\omega$ of the edges of a collection of $2n$ triangles yielding a triangulation $t$. We do not assume yet that $\omega$ is random. We will construct step by step the triangulation $t$ obtained by gluing the edges of the $2n$ triangles two by two according to $\omega$.

More precisely, we will create a sequence $$ S_0 \to S_{1} \to  \dots \to S_{3n} = t$$ of ``combinatorial surfaces''  where $S_{0}$ is simply made of $2n$ disjoint triangles, and where we move on from $ S_{i} $ to $S_{i+1}$ by identifying two edges of the pairing $\omega$. More specifically, $S_{i}$ will be a union of labeled maps with distinguished faces called the holes (they are in light green on Figure \ref{fig_state_exploration}). The holes are made of the edges which are not yet paired. The set of these edges will be called the boundary of the surface and be denoted by $ \partial S_{i}$. Clearly, we have $|\partial S_{0}|=6n$, so for every $1 \leq i \leq 3n$, we have
$$ | \partial S_{i}| =  6n-2i.$$

\begin{figure}[!h]
 \begin{center}
\begin{overpic}[width=15cm]{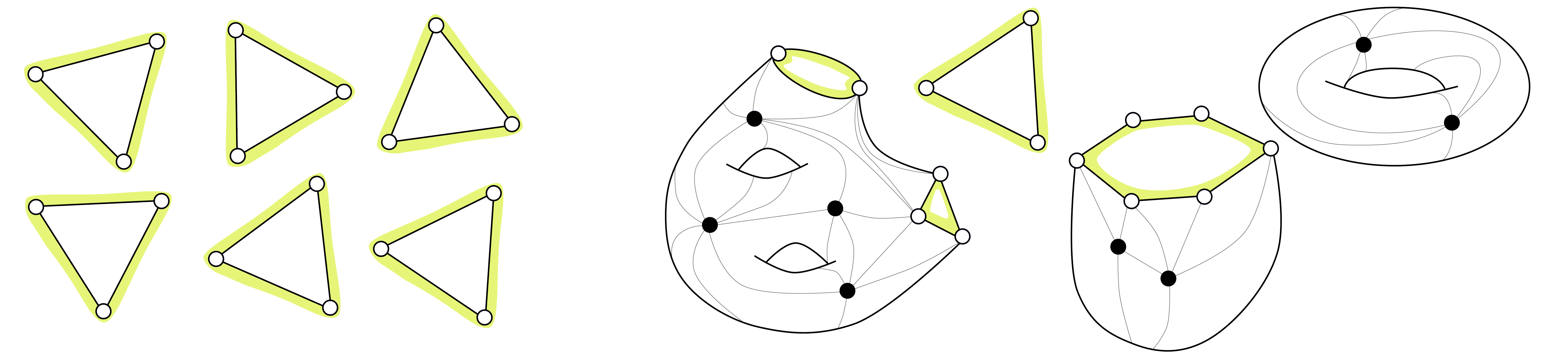}
\put (34,12) {etc.}
\put (98,12) {etc.}
\end{overpic}
 \caption{Starting configuration (on the left) and a typical state of the exploration (on the right). Here and later the labeling of the oriented edges does not appear for the sake of visibility. The final vertices of the triangulation are black dots whereas ``temporary'' vertices are in white. Notice on the right side that $S_{i}$ contains a closed surface without boundary: if this happens, the final surface $S_{n}$ is disconnected.}\label{fig_state_exploration}
 \end{center}
 \end{figure}

To go from $S_{i}$ to $S_{i+1}$, we select an edge on $ \partial S_{i}$ which we call \emph{the edge to peel} and identify it with its partner edge in $\omega$, also belonging to $\partial S_{i}$. A detailed description of each of the cases that may arise when going from $S_i$ to $S_{i+1}$ can be found in \cite[Section 3.1]{BCP19a}. In particular, we call "true vertices" the vertices of $S_i$ that are not on its boundary, and therefore truly correspond to a vertex of $t$, and we call "temporary vertices" the vertices of $S_i$ lying on $\partial S_i$. We recall from \cite{BCP19a} that the only cases where a new true vertex $v$ is created between $S_i$ and $S_{i+1}$ are:
\begin{itemize}
\item[$\bullet$]
if the peeled edge is glued to one of its two neighbours along the same hole (if furthermore the hole has perimeter $2$, then $2$ true vertices are created);
\item[$\bullet$]
if the peeled edge belongs to a hole of perimeter $1$ and is glued to another hole of perimeter $1$.
\end{itemize}
When this occurs, we will also say that the vertex $v$ is \emph{closed} at time $i$.

We now move on to our random setting and apply the above discussion to the case where the gluing $\omega$ is uniform, i.e. $t=T_n$ is a uniform triangulation with $2n$ faces. On top of $\omega$, the sequence $S_{0} \to S_{1} \to \dots \to S_{3n}$ depends on an \emph{algorithm} called the peeling algorithm, which is simply a way to pick the next edge to peel $ \mathcal{A}(S_{i}) \in \partial S_{i}$. Highlighting the dependence in $ \mathcal{A}$, we can thus form the random exploration sequence 
$S_{0}^{ \mathcal{A}} \to S_{1}^{ \mathcal{A}} \to \dots \to S_{3n}^{ \mathcal{A}} =  T_n$ by starting with $S_{0}^{ \mathcal{A}}$, the initial configuration made of the labeled triangles. To go from $ S_{i}^{ \mathcal{A}}$ to $ {S}_{i+1} ^{ \mathcal{A}}$, we perform the identification of the edge $ \mathcal{A}(S_{i})$ together with its partner in the pairing $ \omega$. We recall from \cite[Prop 10]{BCP19a} that when $\omega$ is uniform, then at each step $i$, conditionally on $ S_{i}^{ \mathcal{A}}$ and on $ \mathcal{A}(S_{i}^{ \mathcal{A}})$, the edge $\mathcal{A}(S_{i}^{ \mathcal{A}})$ is glued to a uniformly chosen edge in $\partial S_{i}^{ \mathcal{A}} \backslash \{\mathcal{A}(S_{i}^{ \mathcal{A}})\}$.

The strength of this setup is that, as for planar maps \cite{curienpeeling}, we can use \emph{different} algorithms $ \mathcal{A}$ to explore the \emph{same} random triangulation $T_n$ and then to get \emph{different} types of information. We will see this motto in practice in the following sections. When exploring our random triangulation with a given peeling algorithm, we will always write $( \mathcal{F}_{i})_{0 \leq i \leq 3n}$ for the canonical filtration generated by the exploration.

\subsection{Proof of Proposition \ref{combi_estimate_large_vertices}}

\begin{proof}[Proof of Proposition \ref{combi_estimate_large_vertices}]
%It is enough to prove the result in the model where $T_n$ comes from a uniform pairing edges, without conditioning to be connected.
Conditionally on $T_n$, let $c_1$ and $c_2$ be two uniform corners of $T_n$, and let $f_1, f_2$ and $v_1, v_2$ be the incident faces and vertices. Since there are at most $(2n \times 3)^2$ possible choices of $(c_1,c_2)$, it is enough to prove
\[ \PP \left( \deg(v_1) \deg(v_2) \geq n^{1+\eps} \mbox{ and the face $f'$ does not exist}\right)=o \left( \frac{1}{n^2} \right).\]
Since there are at most $(6n)^2$ possible values of the pair $\left(\deg(v_1), \deg(v_2) \right)$, it is enough to prove that, for any $d_1, d_2$ with $d_1 d_2 \geq n^{1+\eps}$, if we write $\beta_i=\frac{\log d_i}{\log n}$ for $i \in \{1,2\}$, we have
\begin{equation}\label{eqn_find_fprime}
\PP \left(\begin{array}{c} \deg(v_1)=d_1, \deg(v_2)=d_2 \mbox{ and there is no $f'$} \\ \mbox{ such that $d_{v_i}(f_i,f') \leq 3n^{\beta_i/(\beta_1+\beta_2-\eps/2)}$ for $i \in \{1,2\}$} \end{array} \right) =o \left( \frac{1}{n^4} \right).
\end{equation}

The proof of \eqref{eqn_find_fprime} relies on a peeling exploration as the ones defined in Section \ref{subsec_defn_peeling}. For $i \in \{1,2\}$, let $\ell_i=n^{\beta_i/(\beta_1+\beta_2-\eps/2)}$. Note that since $d_1 d_2 \geq n^{1+\eps}$, we have $\beta_1+\beta_2 \geq 1+\eps$, so $\ell_i \leq d_i^{\frac{1}{1+\eps/2}}<\frac{d_i}{4}$ if $n$ (and therefore $d_i$) is large enough. In particular, this implies $\ell_1+\ell_2<\frac{n}{2}$. On the other hand, we have $\beta_i \geq \eps$, which implies 
\begin{equation}\label{eqn_basic_estimate_l}
n^{\eps/2} \leq \ell_i \leq \frac{d_i}{4}.
\end{equation}

%We will consider an exploration $S_0 \to S_1 \to \dots \to S_{3n}$ of $T_n$. 
Note that we can sample $(T_n,f_1,f_2,v_1,v_2)$ as follows. We start from a collection $S_0$ of $2n$ triangles, we pick two triangles $f_1$ and $f_2$ uniformly among them and pick two corners $c_1$ and $c_2$ incident respectively to $f_1$ and $f_2$. We then run a  peeling exploration which keeps track of the faces $f_1,f_2$ and of the corners $c_{1}, c_{2}$. By a slight abuse of notation, we will also call $v_1$ (resp. $v_2$) the vertices  of $S_i$ which are "inherited" from the initial corners $c_1,c_2$. Notice that $v_1,v_2$ stay incident to $c_{1}$ and $c_{2}$ along the exploration.

The peeling algorithm we will use is the following:

\fbox{ \begin{minipage}{14cm}
\begin{itemize}
\item[$\bullet$]
For $0 < i \leq \ell_1$, the peeled edge is an edge incident to $v_1$. The exploration is stopped if the vertex $v_1$ is closed or if the peeled edge is glued to the face $f_2$.
\item[$\bullet$]
For $\ell_1 < i \leq \ell_1+\ell_2$, the peeled edge is an edge incident to $v_2$. The exploration is stopped if the vertex $v_2$ is closed or if the peeled edge is glued to the connected component containing $v_1$.
\end{itemize}
\end{minipage}}
\medskip

Note that the algorithm depends on $d_1$ and $d_2$, and that it makes sense because $\ell_1+\ell_2<3n$, so it will indeed be stopped before everything is explored. We call the exploration \emph{successful} if it is stopped before time $\ell_1+\ell_2$, either by the closure of $v_1$ or $v_2$, or by merging the connected components of $v_1$ and $v_2$. We will show that the probability for the exploration to be successful is $1-o \left( \frac{1}{n^4} \right)$, and that if $n$ is large enough, the success of the exploration ensures that the event of \eqref{eqn_find_fprime} does not occur.

Let us start with a proof that the success of the exploration prevents the event in \eqref{eqn_find_fprime} from happening. First, if the exploration is stopped by the closure of $v_1$, then $v_1$ is incident to less than $\ell_1$ faces, so its degree can be crudely bounded by $3 \ell_1$, which implies $\deg(v_1)<d_1$ if $n$ is large enough. Similarly, if the exploration is stopped by the closure of $v_2$, then $\deg(v_2)<d_2$. Moreover, by construction, any face $f$ lying in the connected component of $v_1$ at time $i \leq \ell_1$ is incident to $v_1$. We claim that $f$ has a corner $c$ with $d_{v_1}(c_1,c) \leq 3\ell_1$. Indeed, the neighbourhood of $v_1$ in the explored part at time $i$ is always a gluing of corners belonging to faces already explored. It is possible that several corners of the same face appear, but the number of corners is bounded by $3i \leq 3\ell_1$, so it is always possible to go from $c_1$ to a corner of $f$ in the neighbourhood of $v_1$ by crossing at most $3\ell_1$ corners.

Therefore, if the exploration is stopped before time $\ell_1$ because $f_2$ is glued to the peeled edge, then $f_2$ is incident to $v_1$ and has a corner $c_2$ with $d_{v_1}(c_1,c_2) \leq 3\ell_1$, so we can take $f'=f_2$. If the exploration is stopped between times $\ell_1$ and $\ell_2$ because the two components of $v_1$ and $v_2$ are glued, let $f'$ be the face incident to $v_1$ which is glued to the peeled edge at the last peeling step. Since the peeled edge is incident to $v_2$, the face $f'$ is incident to both $v_1$ and $v_2$, at two corners $c'_1$ and $c'_2$. Moreover, since $f'$ is in the component of $v_1$ at time $\ell_1$, we have $d_{v_1}(c_1,c'_1) \leq 3 \ell_1$. Finally, by the same reasoning around $v_2$, we also have $d_{v_2}(c_2,c'_2) \leq 3 \ell_2$, so $f'$ satisfies the desired properties and the event in \eqref{eqn_find_fprime} does not occur.

It remains to estimate the probability of non-success of the exploration. The basic idea is the following: we first show that at time $\ell_1$, the number of boundary edges in the component of $v_1$ is of order $\ell_1$. Therefore, at each step $\ell_1 < i < \ell_1+\ell_2$, the probability to finish the exploration by gluing the two components is of order $\frac{\ell_1}{n}$, which will be enough to conclude since $\ell_1 \ell_2$ is much larger than $n$.

More precisely, we recall that for every $i$, we denote by $\kF_i$ the $\sigma$-algebra generated by the first $i$ peeling steps. We also denote by $P_i$ the number of boundary edges of the component of $v_1$ in $S_i$ (these edges may lie on several different holes). Note that $P_0=3$ and that $P_i \leq i+2$ for every $i$. For $i<\ell_1$, the number $P_{i+1}$ is equal to $P_i-2$ if the peeled edge is glued to another boundary edge of the component of $v_1$, and to $P_i+1$ if this is not the case. Therefore, we have
\begin{align*}
\E \left[ P_{i+1}-P_i | \kF_i \right] &= -2 \frac{P_i-1}{6n-2i-1} + \frac{6n-2i-P_i}{6n-2i-1}\\ &=1-3\frac{P_i-1}{6n-2i-1}\\
& \geq 1-3\frac{i+1}{6n-2i-1}\\
& \geq 1-3\frac{\ell_1+1}{6n-2\ell_1-1}\\
& > \frac{1}{2}
\end{align*}
by using in the end the fact that $\ell_1 < \frac{n}{2}$. Since the increments $|P_{i+1}-P_i|$ are bounded by $2$, by the Azuma inequality, we obtain
\[ \PP \left( \mbox{the exploration does not stop before $\ell_1$ and } P_{\ell_1} \leq \frac{\ell_1}{4} \right) \leq \exp \left(-\frac{\ell_1}{128} \right) \leq \exp \left(-\frac{n^{\eps/2}}{128} \right) = o \left( n^{-4} \right),\]
where we used \eqref{eqn_basic_estimate_l} in the end. Therefore, we may assume that if the exploration has not stopped at time $\ell_1$, then $P_{\ell_1} \geq \ell_1/4$. But if this is the case, then for any $\ell_1 < i \leq \ell_1+\ell_2$, we have
\[ \PP \left( \mbox{the peeled edge at time $i$ is glued to the component of $v_1$} | \kF_i \right) = \frac{P_{\ell_1}}{6n-2i-1} \geq \frac{\ell_1}{24n}.\]
If this last event occurs for some $i$, the exploration is stopped and is succesful, so we finally have
\begin{align*}
\PP \left( \mbox{the exploration is not successful} \right) &\leq o \left(n^{-4} \right) + \left( 1-\frac{\ell_1}{24n} \right)^{\ell_2}\\ &\leq o(n^{-4})+\exp \left( -\frac{\ell_1 \ell_2}{24n}\right)\\
&\leq o(n^{-4})+\exp \left( -n^{\eps/3} \right)\\
&=o(n^{-4})
\end{align*}
by using the definition of $\ell_1$ and $\ell_2$.
\end{proof}
\subsection{Proof of Proposition \ref{combi_estimate_small_vertices}}

\begin{proof}[Proof of Proposition \ref{combi_estimate_small_vertices}]
Let $f_1, f_2$ be two uniform independent faces, and let $v_1, v_2$ be the vertices incident to uniformly chosen corners of $f_1$ and $f_2$. Let also $d_1, d_2 \geq n^{2\eps}$ be such that $d_1d_2 \leq n^{1+\eps}$. For the same reasons as in the proof of Proposition \ref{combi_estimate_large_vertices}, it is enough to prove that
\begin{equation}\label{eqn_find_vprime}
\PP \left( \deg(v_1)=d_1, \deg(v_2)=d_2 \mbox{ and neither $f''$ nor $(v',f'_1,f'_2)$ exists} \right)=o \left( \frac{1}{n^4} \right).
\end{equation}
To prove this, we will rely on a peeling algorithm similar to the one used to prove Proposition \ref{combi_estimate_large_vertices}. As previously, we will pick $f_1, f_2, v_1, v_2$ in $S_0$ and follow them along the exploration. 
However, since the two vertices have too small degrees, we will need to find a third vertex $v'$ "inbetween" them, so the algorithm will be more complicated. Basically, we first explore the neighbourhood of $v_1$ until it becomes a true vertex, then the neighbourhood of $v_2$ until it becomes a true vertex, and finally we explore all the neighbours of $v_1$ until one of them is glued to a neighbour of $v_2$. To describe precisely the last phase of the exploration, we will assign colours to some of the vertices: the neighbour of $v_1$ that we are currently exploring will be red, the neighbours of $v_1$ that we can still explore later will be blue, and the neighbours of $v_1$ that we are not allowed to explore anymore will be black. We denote by $\tau_1$ (resp. $\tau_2$) the closure time of $v_1$ (resp. $v_2$). Here is a complete description of the peeling algorithm, which is divided into three phases (see also Figure \ref{fig:algo_prop_5}):

\fbox{ \begin{minipage}{14cm}
\begin{itemize}
\item
Phase 1: exploration of the neighbourhood of $v_1$:
\begin{itemize}
\item
For $0<i \leq \tau_1$, the peeled edge is a boundary edge incident to $v_1$;
\item
if $\tau_1<\frac{d_1}{4}$, the exploration is stopped at time $\tau_1$;
\item
if $\tau_1>d_1$, the exploration is stopped at time $d_1$;
\item
for $i \leq \tau_1$, if the peeling step $i$ glues together the connected components of $v_1$ and $v_2$, then the exploration is stopped at time $i$.
\end{itemize}
\item
Phase 2: exploration of the neighbourhood of $v_2$:
\begin{itemize}
\item
for $\tau_1 < i \leq \tau_2$, the peeled edge is a boundary edge incident to $v_2$;
\item
if $\tau_2-\tau_1<\frac{d_2}{4}$, the exploration is stopped at time $\tau_2$;
\item
if $\tau_2-\tau_1>d_2$, the exploration is stopped at time $\tau_1+d_2$;
\item
for $\tau_1< i \leq \tau_2$, if the peeling step $i$ glues together the connected components of $v_1$ and $v_2$, then the exploration is stopped at time $i$;
\end{itemize}
\item
Phase 3: trying to link $v_1$ to $v_2$:
\begin{itemize}
\item
at time $\tau_2$, we colour in red one of the vertices on the boundary of the connected component of $v_1$, and all the others in blue;
\item
for $i > \tau_2$, we peel the boundary edge on the left of the red vertex;
\item
for $i > \tau_2$, if the red vertex has been red for at least $\frac{n^{1+\eps}}{d_1 d_2}$ steps, we colour it in black, and choose a blue vertex that we colour in red;
\item
for $i > \tau_2$, if a blue or red vertex is glued to another blue or red vertex or to the peeled edge at time $i$, we colour it in black;
\item
for $i > \tau_2$, if the peeling step $i$ glues together the connected components of $v_1$ and $v_2$, then the exploration is stopped at time $i$;
\item
if there is no more blue or red vertex, the exploration is stopped and declared unsuccessful.
\end{itemize}
\end{itemize}
\end{minipage}}
\medskip

\begin{figure}
\includegraphics[scale=1]{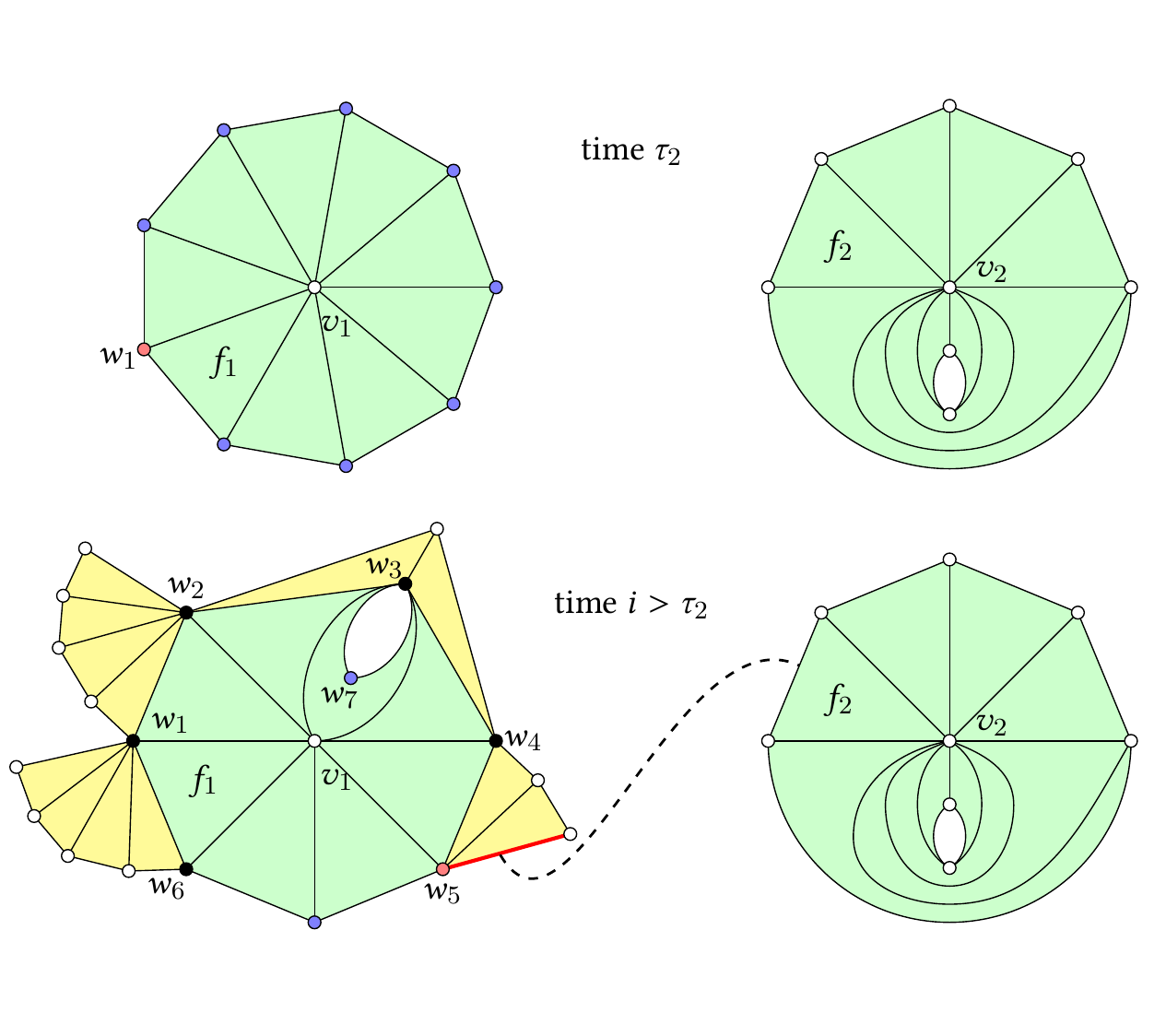}
\caption{The peeling algorithm used to prove Proposition \ref{combi_estimate_small_vertices}. On top, the connected components containing $v_1$ and $v_2$ at time $\tau_2$. The first red vertex is $w_1$. On the bottom, the components at some time $i>\tau_2$. The black vertices are those whose neighbourhood has been explored during too much time ($w_1$, $w_2$), or those which have been affected by the exploration before becoming red ($w_3$, $w_4$, $w_6$). Note that $w_7$ is still blue. The current red vertex is $w_5$. If the peeled edge (in red) is glued to the component of $v_2$, the exploration will be stopped succesfully, with $w_5$ playing the role of the vertex $v'$ of Proposition \ref{combi_estimate_small_vertices}.}\label{fig:algo_prop_5}
\end{figure}

When the exploration is stopped, in all the cases except the last one, it is declared successful. As in the proof of Proposition \ref{combi_estimate_large_vertices}, we will first prove that if the exploration is successful, then the event of \eqref{eqn_find_vprime} does not occur, and then that the probability of success is $1-o(n^{-4})$.

First, just like in the proof of Proposition \ref{combi_estimate_large_vertices}, if the exploration is stopped because we do not have $\frac{d_1}{4} \leq \tau_1 \leq d_1$, then $\deg (v_1) \ne d_1$ so the event of \eqref{eqn_find_vprime} does not occur. Similarly, if we do not have $\frac{d_2}{4} \leq \tau_2-\tau_1 \leq d_2$, then $\deg(v_2) \ne d_2$. Second, if the exploration is stopped at some time $i \leq \tau_2$ because the components of $v_1$ and $v_2$ are glued together, then either an edge incident to $v_1$ is glued to a face incident to $v_2$, or the vertex $v_2$ is glued to a neighbour of $v_1$. In both cases, the vertices $v_1$ and $v_2$ are neighbours in $S_i$, so they are neighbours in $T_n$. Hence, we can take as $f''$ a face incident to an edge between $v_1$ and $v_2$.

Therefore, the only case left to treat is the one where the components of $v_1$ and $v_2$ are glued together at time $i>\tau_2$. In this case, let $v'$ be the red vertex in $S_{i-1}$. Since $v'$ has been blue at some point, there is a face $f'_1$ in $S_{i-1}$ incident to both $v'$ and $v_1$. Moreover, let $f'_2$ be the face of the component of $v_2$ in $S_{i-1}$ to which the peeled edge at step $i$ has been glued. Then $f'_2$ is incident to both $v'$ (since $v'$ is on the peeled edge at step $i$) and to $v_2$ (the component of $v_2$ in $S_{i-1}$ contains only faces incident to $v_2$) in $S_i$, so it is also true in $T_n$. Therefore, we only need to make sure that $d_{v'}(c'_1,c'_2) \leq \frac{n^{1+2\eps}}{d_1 d_2}$ in $S_i$, where $c'_1$ and $c'_2$ are two corners of $v'$ in $S_i$ incident to $f'_1$ and $f'_2$. Let $i'$ be the step at which $v'$ has become red for the first time. Then $v'$ was blue in $S_{i'}$, so there are only two faces incident to $v'$ in $S_{i'}$ (the two faces incident to the edge from $v'$ to $v_1$). Moreover, by the definition of our algorithm we have $i-i' \leq \frac{n^{1+\eps}}{d_1 d_2}$, i.e. $v'$ may only remain red during at most $\frac{n^{1+\eps}}{d_1 d_2}$ steps. Since each step between $i'$ and $i$ adds at most one face incident to $v'$, there are at most $\frac{n^{1+\eps}}{d_1 d_2}+2$ faces of $S_i$ incident to $v'$, so at most
\[3 \left( \frac{n^{1+\eps}}{d_1 d_2}+2 \right) \leq  \frac{n^{1+2\eps}}{d_1 d_2} \]
corners incident to $v'$. Therefore, we have $d_{v'}(c'_1, c'_2)  \leq \frac{n^{1+2\eps}}{d_1 d_2}$ in $S_i$, and this is also true in $T_n$.

We now prove that the probability for the exploration to be unsuccessful is $o(n^{-4})$. Roughly speaking, we want to prove that there will be many possible successive red vertices during phase 3, and that each of them has a reasonable chance to stop the exploration successfully. Therefore, we will need to bound from below the number of blue vertices in $S_{\tau_2}$.

We first estimate the number of steps needed to know if the exploration is successful. The number of blue vertices in $S_{\tau_2}$ is bounded by the boundary length of the component of $v_1$ in $S_{\tau_2}$, which is at most $d_1$. Moreover, during phase 3 of the exploration, the number of blue vertices never increases, and it decreases at least every $\frac{n^{1+\eps}}{d_1 d_2}$ steps. Hence, the total duration of the exploration is bounded by
\[ d_1+d_2+d_1 \times \frac{n^{1+\eps}}{d_1 d_2} \leq 3n^{1-\eps}\]
since $d_1 \leq \frac{n^{1+\eps}}{d_2} \leq \frac{n^{1+\eps}}{n^{2\eps}}=n^{1-\eps}$. In particular, for every step $i$ of the exploration, the number of boundary edges of $S_i$ which do not belong to the components of $v_1$ and $v_2$ is larger than $6n-3n^{1-\eps}$, which is $\geq 5n$ if $n$ is large enough.

Let $B_i$ be the number of boundary vertices of the connected component of $v_1$ in $S_i$. The last discussion implies that for every $i \leq \tau_1$, we have
\[ \PP \left( B_{i+1}=B_i+1 | \mathcal{F}_i \right) \geq \frac{5}{6}. \]
By the same argument based on the Azuma inequality as in the proof of Proposition \ref{combi_estimate_large_vertices}, we have 
\[ \PP \left( \tau_1 \geq \frac{d_1}{4} \mbox{ and } B_{\tau_1}<\frac{\tau_1}{4} \right) =o(n^{-4}).\]
Therefore, we may assume $B_{\tau_1} \geq \frac{\tau_1}{4} \geq \frac{d_1}{16}$. If this occurs and the exploration does not end succesfully before $\tau_2$, then we have at least $\frac{d_1}{16}$ blue vertices in $S_{\tau_2}$. Similarly, we may assume that the number of boundary edges of the component of $v_2$ at time $\tau_2$ is at least $\frac{d_2}{16}$.

We now estimate the total number of blue vertices that become black without being red before the end of the exploration (because of the fourth item of Phase 3 in the definition of the peeling algorithm). A blue vertex $v \in S_i$ may be turned black in $S_{i+1}$ for three different reasons:
\begin{enumerate}
\item\label{black_reason_i}
the peeled edge is glued at time $i$ to one of the two boundary edges incident to $v$;
\item\label{black_reason_ii}
the peeled edge is glued at time $i-1$ to the edge at distance $1$ on the right of $v$ along the boundary, so that $v$ is the second end of the peeled edge at time $i$;
\item\label{black_reason_iii}
the red vertex changes at time $i$, and the new red vertex is the one on the right of $v$, so $v$ is the other end of the peeled edge.
\end{enumerate}
At each step, at most $2$ vertices may be turned black for reason \ref{black_reason_i} and $1$ for reason \ref{black_reason_ii}.
Hence, the number of vertices that are turned black for reasons \ref{black_reason_i} and \ref{black_reason_ii} is at most three times the number of times after $\tau_2$ when the peeled edge is glued to an edge at distance at most $2$ from a blue vertex along the boundary.

For every $i$, the probability for this to occur at step $i$ conditionally on $\kF_i$ is at most $5 \times \frac{d_1}{5n}$ (since the number of blue vertices is bounded by $d_1$).
Since the total number of peeling steps is at most $n^{1-\eps}$, the expected number of times where this occurs is bounded by $\frac{d_1}{n^{\eps}}$. By using the Azuma inequality as before, we can also show that the probability that this occurs more than $2\frac{d_1}{n^{\eps}}$ times is $o(n^{-4})$, so
\[ \PP \left( 6\frac{d_1}{n^{\eps}} \mbox{ blue vertices become black for reasons \ref{black_reason_i} and \ref{black_reason_ii}} \right) =o (n^{-4}).\]
Since there are at least $\frac{d_1}{16}$ blue vertices at time $\tau_1$, with probability $1-o(n^{-4})$, at least $\frac{d_1}{17}$ of them either are coloured red at some point, or remain blue until the end, or are turned black for reason \ref{black_reason_iii}. Moreover, at most half of these vertices can be turned black for reason \ref{black_reason_iii} because we can only turn one vertex black in this way everytime there is a new red vertex. Therefore, at least $\frac{d_1}{50}$ vertices will either be red at some point, or remain blue until the end of the exploration.

Finally, recall that with probability $1-o(n^{-4})$, the total boundary length of the component of $v_2$ is larger than $\frac{d_2}{16}$. If this event occurs, then at each peeling step $i>\tau_2$, the conditional probability (on $\kF_i$) to complete the exploration in a succesful way by gluing the components of $v_1$ and $v_2$ is at least $\frac{1}{6n} \times \frac{d_2}{16}>\frac{d_2}{100n}$.
Moreover, if the exploration fails, we know that with probability $1-o(n^{-4})$, at least $\frac{d_1}{50}$ red vertices have been "investigated", each one during $\frac{n^{1+\eps}}{d_1 d_2}$ steps, so there have been at least $\frac{d_1}{50} \times \frac{n^{1+\eps}}{d_1 d_2}$ "failed" steps after $\tau_2$. Therefore, we have
\[ \PP \left( \mbox{the exploration is not successful} \right) \leq o(n^{-4}) + \left( 1-\frac{d_2}{100n} \right)^{n^{1+\eps}/(50d_2)} \leq \exp \left( -\frac{n^{\eps}}{5000} \right)+o(n^{-4}),\]
which concludes the proof.
\end{proof}

\subsection{Proof of Proposition \ref{combi_estimate_tiny_vertices}}

\begin{proof}[Proof of Proposition \ref{combi_estimate_tiny_vertices}]
As in the proofs of Propositions \ref{combi_estimate_large_vertices} and \ref{combi_estimate_small_vertices}, let $v$ be the vertex incident to a uniform corner of a uniform face of $T_n$. We will prove that, with probability $1-o(n^{-1})$, either $T_n$ is disconnected, or there is a vertex $v'$ with degree at least $n^{1/4}$ at graph distance at most $6$ from $v$. Since we know that $T_n$ is connected with probability $1-o(1)$, this is enough to guarantee\footnote{Alternatively, we could also be more precise in what follows to show that the probability to disconnect $v$ is $O(n^{-2})$, but this would make the proof longer.} the conclusion of Proposition \ref{combi_estimate_tiny_vertices} with probability $1-o(1)$.

Like the algorithm used in the proof of Proposition \ref{combi_estimate_small_vertices}, the peeling algorithm we will use to prove this depends on a vertex coloured in red on the boundary. This red vertex is roughly the candidate for $v'$ that we are currently testing. Here is the definition of the algorithm:

\fbox{ \begin{minipage}{14cm}
\begin{itemize}
\item
the red vertex in $S_0$ is $v$;
\item
if the red vertex of $S_{i-1}$ is closed at step $i$, we choose a new red vertex on $\partial S_i$ at minimal graph distance from $v$;
\item
at each step, the peeled edge is the edge on the left of the red vertex along the boundary;
\item
if all the connected component containing $v$ is closed (so that there is no possible choice for the new red vertex), the exploration is stopped and declared successful;
\item
if there have been $n^{1/4}$ consecutive peeling steps without a closure time, the exploration is stopped and declared successful;
\item
if $3$ closure times have occured, the exploration is stopped and declared unsuccessful.
\end{itemize}
\end{minipage}}
\medskip

We first note that if the exploration is successfully stopped because the connected component of $v$ has no boundary anymore, then $T_n$ must be disconnected, which is one of the two conclusions we are trying to reach.

We now study the case where it is stopped successfully by $n^{1/4}$ consecutive steps without any closure. We first note that if at time $i$ we choose a new red vertex $v^* \in S_i$, there is a graph geodesic $\gamma$ between $v$ and $v^*$ in $S_i$. All the vertices on $\gamma$ are closer to $v$ than $v^*$, so by the definition of $v^*$ they must be closed vertices. Since at most $3$ closure times have occured up to step $i$ and each has closed at most $2$ vertices, the length of $\gamma$ is at most $6$, so the graph distance between $v$ and the red vertex in $S_i$ is always at most $6$. Therefore, if the exploration is successfully stopped at time $i$, the last red vertex $v'$ is at distance at most $6$ from $v$ in $S_i$, so it is also the case in $T_n$. Moreover, the vertex $v'$ has been red for $n^{1/4}$ steps, so it is incident to at least $n^{1/4}$ corners in $S_i$ and therefore also in $T_n$, so it has degree at least $n^{1/4}$ in $T_n$, so $v'$ satisfies the conclusion of the proposition. Therefore, it is enough to show that the probability for the exploration to fail is $o(n^{-1})$.

But if the exploration fails, then there are at least $3$ closure times during the first $3n^{1/4}$ peeling steps. We recall from \cite[Section 3.1]{BCP19a} that there are two ways in which $i$ may be a closure time:
\begin{itemize}
\item
if the peeled edge at time $i$ is glued to one of its neighbours along the boundary;
\item
if the peeled edge is a loop (i.e. a hole of perimeter $1$) and is glued to another loop.
\end{itemize}
Hence, for every $i$, we have
\[ \PP \left( \mbox{$i$ is a closure time} | \kF_i \right) \leq \frac{2+L_i}{2n-2i-1},\]
where $L_i$ is the number of boundary loops at time $i$. We can now bound $L_i$ in a (much) cruder way than in \cite{BCP19a}. Each peeling step creates at most two loops, so for $i \leq 3n^{1/4}$, we have $L_i \leq 6n^{1/4}$. Therefore, we have
\[ \PP \left( \mbox{$i$ is a closure time} | \kF_i \right) \leq \frac{6n^{1/4}+2}{2n-6n^{1/4}-1} \leq 6n^{-3/4}\]
if $n$ is large enough.
Therefore, for any $0 \leq i<j<k \leq 6n^{1/4}$, the probability that $i,j,k$ are all closure times is at most $\left( 6n^{-3/4} \right)^3=216 n^{-9/4}$. By summing over all triples $(i,j,k)$, we obtain
\[ \PP \left( \mbox{there are $3$ closure times in the first $3n^{1/4}$ steps} \right) \leq \frac{(3 n^{1/4})^3}{6} \times 216 n^{-9/4}=O(n^{-3/2})=o(n^{-1}),\]
which proves that the exploration is successful and concludes the proof.
\end{proof}

%%%%%%%%%%%%%%%%%%%%%%%%%%%%%%%%%%%%%%%%%%%%%%%%%
%		R E F E R E N C E S
%%%%%%%%%%%%%%%%%%%%%%%%%%%%%%%%%%%%%%%%%%%%%%%%%
%\nocite{*}
\bibliographystyle{alpha}
\bibliography{bib_hyperbolic.bib}

\end{document}